\documentclass{article}

\pdfoutput=1

\usepackage{amsmath, amssymb, amsthm, enumerate, fullpage}
\usepackage{verbatim, enumitem, bbm, etoolbox}

\newcommand{\m}{\mathcal}
\renewcommand{\b}{\mathbb}

\apptocmd{\lim}{\limits}{}{}

\newcommand{\im}{\textrm{im}}

\newcommand{\tp}{\textrm{tp}}
\renewcommand{\o}{\overline}

\renewcommand{\th}{\textrm{Th}}

\newcommand{\LO}{\textrm{LO}}
\newcommand{\IT}{\textrm{IT}}

\newcommand{\proves}{\vdash}

\theoremstyle{plain}
\newtheorem{thm}{Theorem}
\newtheorem{theorem}[thm]{Theorem}

\newtheorem{lemma}[thm]{Lemma}

\newtheorem{corollary}[thm]{Corollary}
\newtheorem{prop}[thm]{Proposition}
\newtheorem{defthm}[thm]{Definition / Theorem}
\newtheorem{proposition}[thm]{Proposition}

\numberwithin{thm}{section}

\numberwithin{subcase}{case}

\theoremstyle{definition}
\newtheorem{definition}[thm]{Definition}

\def\Ind{\setbox0=\hbox{$x$}\kern\wd0\hbox to 0pt{\hss$\mid$\hss}
	\lower.9\ht0\hbox to 0pt{\hss$\smile$\hss}\kern\wd0}

\def\Notind{\setbox0=\hbox{$x$}\kern\wd0\hbox to 0pt{\mathchardef
		\nn=12854\hss$\nn$\kern1.4\wd0\hss}\hbox to 0pt{\hss$\mid$\hss}\lower.9\ht0
	\hbox to 0pt{\hss$\smile$\hss}\kern\wd0}

\newcommand{\Mod}{\textrm{Mod}}
\newcommand{\iso}{\cong}
\newcommand{\ZFC}{\textrm{ZFC}}
\newcommand{\TY}{\textrm{TY}}
\newcommand{\otp}{\textrm{otp}}

\newcommand{\boreleq}{\sim_{_{\!\!B}}}
\newcommand{\borelleq}{\leq_{_{\!B}}}
\newcommand{\borellt}{<_{_{\!B}}}

\newcommand{\class}{/\!\!\sim}

\begin{document}	
	\bibliographystyle{plain}
	
	\title{The Complexity of Isomorphism for Complete Theories of Linear Orders With Unary Predicates}

	\author{Richard Rast\thanks{The author is greatly indebted to the anonymous Reviewer \#1 for numerous improvements for readability and notation. The author was partially supported by NSF Research Grant DMS-1308546. }}
	
	\date{\today}

	\maketitle
	
	\begin{abstract}
		Suppose $A$ is a linear order, possibly with countably many unary predicates added.  We classify the isomorphism relation for countable models of $\th(A)$ up to Borel bi-reducibility, showing there are exactly five possibilities and characterizing exactly when each can occur in simple model-theoretic terms.  We show that if the language is finite (in particular, if there are no unary predicates), then the theory is $\aleph_0$-categorical or Borel complete; this generalizes a theorem due to Schirmann in \cite{SchirmannLinearOrders}.
	\end{abstract}

	\bigskip
	
	\section{Introduction}
	
	In 1973, Matatyahu Rubin published his master's thesis on the model theory of complete theories of linear orders, possibly with countably many unary predicates added.  Most prominently, he proved in \cite{RubinTheoriesOfLinearOrder} that such a theory has either finitely many or continuum-many countable models, up to isomorphism.  This was part of a larger set of results in his master's thesis, wherein he investigated a huge variety of model-theoretic properties of such theories, such as the size of type spaces, finite axiomatizability, and characterizing saturation of models.
	
	We continue his investigation here, examining what we will call colored linear orders.
	
	\begin{definition}
		Say $0\leq\kappa\leq\aleph_0$, and let $L_\kappa$ be the language $\{<\}\cup\{P_i:i\leq\kappa\}$ where $<$ is a binary relation and each $P_i$ is a unary relation.
		
		A \emph{colored linear order}, or CLO, is a complete $L_\kappa$-theory for some $\kappa$ making $<$ a linear order.
	\end{definition}
	
	We will also refer to a \emph{structure} $A$ as a CLO if its complete theory is a CLO in the above sense.  This should never cause confusion.
	
	We look into the complexity of such theories from two perspectives -- the Borel complexity of isomorphism for countable models of CLOs, and the number of models (of any cardinality) up to back-and-forth equivalence.  Surprisingly, there are essentially five classes of such theories.  First is the $\aleph_0$-categorical theories; then those with finitely many countable models; then those whose complexity corresponds approximately to ``real numbers;'' then those whose complexity corresponds approximately to ``sets of reals numbers;'' then those with unbounded complexity.  With the exception of ``finite,'' each of these classes contains exactly one element up to reducibility, and the Borel complexity lines up exactly with the corresponding count of back-and-forth inequivalent models.  This theorem is finally stated and proved precisely in Theorem~\ref{MainTheorem}.  It is worth noting that these five complexity classes are essentially identical to those appearing for o-minimal theories, as shown in \cite{RastSahotaBCOMinimal}, and for essentially the same reasons -- a divide on local simplicity or nonsimplicity, then a type-counting argument in the simple case.
	
	The outline of the paper is as follows.  We begin by highlighting what is the core of Rubin's work in \cite{RubinTheoriesOfLinearOrder}, since this paper relies heavily on his work there.  We then introduce other background the reader will need, such as notions of sum and shuffle, Rosenstein's characterization of $\aleph_0$-categorical linear orders, and the formal notions of Borel complexity and back-and-forth equivalence needed to make the preceding paragraph rigorous.
	
	In Section~\ref{SelfAdditiveSection} we re-introduce the notion of self-additive CLOs (approximately those which cannot be definably divided into convex pieces) and show they are either minimally or maximally complex.  In Section~\ref{IntervalTypesSection}, we show that CLOs can be definably decomposed into essentially self-additive pieces, and that if any of these are maximally complicated, so is the whole theory.  If not, we characterize back-and-forth equivalence for such theories as fairly simple, showing a strong dichotomy.  We then fine-tune this analysis to give the exact cases which a CLO can fall into, and prove our characterization.
	
	We end with a special case, showing that none of the ``middle cases'' can happen if the language is finite.  This generalizes a theorem of Schirmann in \cite{SchirmannLinearOrders}, where a similar result was shown for complete theories of linear orders.

	\section{Background}
	
	For this section we cover several classical topics which are essential to the study of linear orders, such as convex sums, shuffles, and Rosenstein's characterization of $\aleph_0$-categorical linear orders. But first and foremost, we want to highlight the following ``technical lemma'' of Rubin, which appears as Corollary~2.3 in \cite{RubinTheoriesOfLinearOrder}:
	
	\begin{lemma}\label{RelativizationLemma}
		Let $A$ be a CLO, and let $B\subset A$ be convex.  Let $\phi(\o x)$ be a formula, possibly with parameters from $A\setminus B$.  Then there is a formula $\phi^{\#}(\o x)$ with no parameters where, for all $\o b$ from $B$, $B\models\phi^{\#}(\o b)$ if and only if $A\models \phi(\o b)$.
	\end{lemma}
	
	This is perhaps the reason that CLOs are so nice from a logical perspective.  Because $B$ is convex, the order type of some $b\in B$ and some $a\in A$ is determined by $a$ and the fact that $b\in B$; that is, for any $b,b'\in B$, $b<a$ if and only if $b'<a$.  The rest of the atoms are unary, so hold in $B$ exactly as they would in $A$.  So by an inductive argument, we get the above lemma.
	
	This is used to tremendous effect throughout \cite{RubinTheoriesOfLinearOrder}, primarily to prove that given some CLOs $A\subset B$, we can conclude $A\prec B$\footnote{Here and throughout, $\prec$ represents elementary substructure, following \cite{MarkerMT} for example.}.  We will cite numerous lemmas from \cite{RubinTheoriesOfLinearOrder} which are of this form, and their proofs are all essentially of this form.  We do not reproduce these arguments here, though we do need to produce one ourselves for Lemma~\ref{LocallySimpleManyModelsLemma}, so that the reader can get some of the flavor.  It is our opinion that all of our results on CLOs hinge on two points: the ease of constructing models through sums, and some form of Lemma~\ref{RelativizationLemma}.

	\subsection{Sums and Shuffles}
	
	We now introduce two classical operations, the sum and the shuffle, which go back at least to Hausdorff.  Due to the absence of ``prime models over sets'' in general, we will rely on these operations to construct new models of our theories.  We first examine the notion of a \emph{sum}; if $(I,<)$ is a linear order and for each $i$, $A_i$ is a CLO in the language $L$, we can define $\sum_i A_i$ in the natural way.  It has universe $\{(a,i):a\in A_i, i\in I\}$.  We say $(a,i)<(b,j)$ if $i<j$, or if $i=j$ and $a<b$ in $A_i$.  For any color $P$ in $L$, we say $P(a,i)$ holds in the sum if $P(a)$ holds in $A_i$.  This is an extremely well-behaved operation, and the following properties can be verified immediately (or see \cite{RosensteinLinearOrderings}):
	
	\begin{prop}\label{BasicSumProp}
		Let $(I,<)$ be a linear order and let $(A_i:i\in I)$ be CLOs in the same language $L$.
		
		\begin{enumerate}
			\item If $A_i\iso B_i$ for all $i$, then $\sum_i A_i\iso \sum_i B_i$.
			\item If $A_i\equiv B_i$ for all $i$, then$ \sum_i A_i\equiv \sum_i B_i$.
		\end{enumerate}
	\end{prop}
	
	We use the familiar notation $A_1+\cdots+A_n$ for finite sums.  If $C\subset A$ is convex, then $A$ decomposes as a sum $B_1+C+B_2$, where $B_1$ is the set of elements below every element of $C$, and likewise with $B_2$.  Note that $B_1$ or $B_2$ (or both) may be empty.
	
	\begin{proposition}\label{DefinableSumProp1}
		Suppose $\Phi(x)$ is a partial type\footnote{That is, a set of formulas without parameters with at most $x$ free.}, $A$ is a CLO, and $C=\Phi(A)$ is convex.  Decompose $A$ as $B_1+C+B_2$.  If $D\equiv C$, then define $A_D$ as $B_1+D+B_2$.  Then $A_D\equiv A$ and $\Phi(A_D)=D$.
	\end{proposition}
	\begin{proof}
		First, add a new predicate $P$ to the language, and let $P(a)$ hold for some $a\in A$ if and only if $a\in C$.  Expand $D$ to the new language to let $P$ hold everywhere, and expand $A_D$ to make $P$ true only on $D$.  We show that for all tuples $\o b_1$ and $\o b_2$ from $B_1$ and $B_2$ respectively, $(A,\o b_1,\o b_2)\equiv (A_D,\o b_1,\o b_2)$ in the expanded language.  This is done by an Ehrenfeucht-Fra\"isse game argument. So as usual we may assume the language is finite, fix an $n\in\omega$, and describe a strategy for the second player to win the game of length $n$. Since $C\equiv D$, fix a winning strategy for the second player in the game of length $n$ between $C$ and $D$.  Then for any play, if the first player plays an element of $B_1$ or $B_2$ from one model, the second player plays the same element in the other model.  If the first player plays within $C$ or $D$, the second player follows the winning strategy for those two.  This is well-defined and clearly preserves colors and $<$ within components.  Since the components are convex and we stay within them, this preserves $<$ generally, so proves the result.
		
		That $A\equiv A_D$ follows immediately.  To see that $\Phi(A_D)=D$, first note that $A\models \forall x(P(x)\to \phi(x))$ for all $\phi\in \Phi$, so $D\subset \Phi(A_D)$.  On other hand, for any $b\in A\setminus C$, there is a $\phi\in\Phi$ where $A\models \lnot\phi(b)$.  Since $(A,b)\equiv (A_D,b)$, $A_D\models\lnot\phi(b)$, so $\Phi(A_D)\subset D$, proving the proposition.
	\end{proof}
	
	Next we define the shuffle.  To do this, fix a natural number $n$, and form a countable structure $\m D_n$ in the language $L_n=\{<,P_1,\ldots,P_n\}$ satisfying the following axioms:
	
	\begin{itemize}
		\item $<$ is a linear order which is dense and without endpoints.
		\item The $P_i$ are disjoint, dense, codense, and exhaustive.
	\end{itemize}
	
	It is easy to see that these axioms are consistent and $\aleph_0$-categorical (hence complete), so $\m D_n$ is defined up to isomorphism.  Now for any language $L$ and any CLOs $A_1,\ldots,A_n$, we form the \emph{shuffle} $\sigma(A_1,\ldots,A_n)$ as follows.  For each $i\in \m D_n$, define $D_i$ as $A_j$ if and only if $P_j(i)$ holds.  Then $\sigma(A_1,\ldots,A_n)$ is the sum $\sum_i D_i$.  The following facts are easily verified:
	
	\begin{proposition}\label{GeneralShuffleProposition}
		Let $A_1,\ldots,A_n$ be countable CLOs in the same language $L$.  Then all the following hold:
		
		\begin{enumerate}
			\item If $\tau$ is a permutation of $\{1,\ldots,n\}$, then $\sigma(A_1,\ldots,A_n)\iso \sigma(A_{\tau(1)},\ldots, A_{\tau(n)})$.
			\item If for all $i$, $A_i\equiv B_i$, then $\sigma(A_1,\ldots,A_n)\equiv \sigma(B_1,\ldots, B_n)$.
			\item If for all $i$, $A_i\iso B_i$, then $\sigma(A_1,\ldots,A_n)\iso \sigma(B_1,\ldots,B_n)$.
		\end{enumerate}
	\end{proposition}
	
	While the shuffle may seem somewhat arbitrary, it is important in Rosenstein's characterization of $\aleph_0$-categorical CLOs, and will come up in a natural way in Section~\ref{SelfAdditiveSection}.

	\subsection{$\aleph_0$-categorical Theories}
	
	By convention, we will refer to a \emph{structure} (of any size) as $\aleph_0$-categorical if and only if its complete theory has a unique countable model up to isomorphism.  Following \cite{RosensteinCategoricalOrders}, we will also consider \emph{finite} structures (and their complete theories) to be $\aleph_0$-categorical.
	
	In Section~\ref{IntervalTypesSection}, we will make important use of Rosenstein's characterization of $\aleph_0$-categorical linear orders in \cite{RosensteinCategoricalOrders}, which was extended to CLOs by Mwesigye and Truss in \cite{MwesigyeTrussCategoricalColoredOrders}.  One begins by defining several classes, which we call $\m M_n$.
	
	\begin{itemize}
		\item $\m M_0$ is the set of all one-point CLOs; the colors can be arbitrary.
		\item $\m M_{n+1}$ is the smallest class of CLOs such that all the following are satisfied:
		
		\begin{itemize}
			\item If $A\in\m M_n$, then $A\in \m M_{n+1}$.
			\item If $A,B\in\m M_n$, then $A+B\in\m M_{n+1}$.
			\item If $A_1,\ldots,A_k\in \m M_n$, then $\sigma(A_1,\ldots,A_k)\in\m M_{n+1}$.
		\end{itemize}
		\item $\m M$ is the union $\bigcup_n \m M_n$.
	\end{itemize}
	
	The characterization is:
	
	\begin{thm}[Rosenstein;  Mwesigye, Truss]\label{Aleph0CatCharacterizationThm}
		Let $T$ be a CLO.  Then $T$ is $\aleph_0$-categorical if and only if $T=\th(A)$ for some $A\in\m M$.
	\end{thm}
	
	Note that the above does make sense and is true even if the language is infinite, and we will rely on that.  However, it is nearly vacuous -- if a CLO is $\aleph_0$-categorical, only finitely many of its colors are inequivalent.
	
	Following \cite{RosensteinLinearOrderings}, we can also define a rank; if $A$ is an $\aleph_0$-categorical CLO, let $r(A)$ be the least $n$ where there is some $B\in\m M_n$ such that $A\equiv B$. This turns out to be a useful inductive tool, allowing us to prove all the following facts:
	
	\begin{proposition}\label{Aleph0CatGeneralFactsProp}
		Let $A$ be a (possibly uncountable) CLO in a language $L$.
		
		\begin{enumerate}
			\item If $A$ is $\aleph_0$-categorical, then every convex subset $B\subset A$ is also $\aleph_0$-categorical.  Indeed, $r(B)\leq 2\cdot r(A)+1$.
		\end{enumerate}
		
		If $L$ is finite, then we also get the following:
		
		\begin{enumerate}
			\setcounter{enumi}{1}
			\item For any $n\in\omega$, there are only finitely many $\aleph_0$-categorical CLOs in $L$ of rank $n$.
			\item For any $\aleph_0$-categorical $A$, there are only finitely many convex subsets of $A$, up to back-and-forth equivalence.  This bound is uniform in $r(A)$.
		\end{enumerate}
	\end{proposition}
	\begin{proof}
		(1) First, assume $A$ is countable; we will generalize in a moment.  We show this by induction on rank.  It is trivially true for $r(A)=0$.  So let $r(A)=n+1$.  Then either $A=B_1+B_2$ for some $B_i\in\m M_n$, or $A=\sigma(B_1,\ldots,B_k)$ for some $B_i\in \m M_n$.  In the first (sum) case, if $C\subset B_1+B_2$ is convex, then $C=(B_1\cap C)+(B_2\cap C)$, where each $B_i\cap C$ is a convex subset of the $B_i$.  By induction, $r(B_i\cap C)\leq r(B_i)+1\leq n+1$, so $C$ is the sum of two CLOs with rank at most $n+1$, so $r(C)\leq n+2\leq 2(n+1)+1$, as desired.
		
		In the other (shuffle) case, if $C\subset \sigma(B_1,\ldots,B_k)$ is convex, then $C$ is either $B_i\cap C$ for some $i$, or $(B_{i_1}\cap C)+\sigma(B_1,\ldots,B_k)+(B_{i_2}\cap C)$, where either of the $B_{i_j}$ could be empty.  This is because the left ``edge'' of $C$ either slips exactly between $B_i$ components or cuts into one (corresponding to $B_{i_1}$ being empty or some $B_i$, respectively).  Similarly with the right ``edge.''  If these cut into the same $B_i$ component, there is no shuffle and $C$ is a convex subset of $B_i$, so has rank at most $2 r(B_i)+1\leq 2n+1$.  If they cut into different components, there is an isomorphic copy of the shuffle of the respective $B_i$.  The shuffle has rank $n+1$, while each of the sides has at most $2n+1$, so $r(C)\leq (2n+1)+1+1=2n+3=2(n+1)+1$, as desired.
		
		For the case when $A$ may be uncountable, let $C\subset A$ be convex, and let $(A,C)$ be the structure with an unary predicate for $C$.  Let $(A_0,C_0)\prec (A,C)$ be countable, noting that $C\equiv C_0$, $A\equiv A_0$, and $C_0$ is a convex subset of $A_0$.  Then the preceding special case applies to $(A_0,C_0)$, and by elementary equivalence, the result for $A_0$ and $C_0$ implies it for $A$ and $C$, as desired.
		
		(2) If there are $k$ distinct unary predicates in $L$, there are $2^k$ one-point CLOs, so there are $2^k$ elements of $\m M_0$.  If $\m M_n$ has $m$ elements, then $\m M_{n+1}$ has $m$ elements from $\m M_n$, $m^2$ elements as sums from $\m M_n$, and $\sum_{i=1}^m \binom{m}{i}$ elements as shuffles from $\m M_n$.  So $\m M_{n+1}$ is finite, as desired.
		
		(3) If $B$ is a convex subset of some $A$ with $r(A)\leq n$, then $r(B)\leq 2n+1$ by (1).  By (2), there is a finite number of $\aleph_0$-categorical CLOs of rank at most $2n+1$, and this depends only on $n$.
	\end{proof}
	
	Finally, we include Corollary~5.11 of \cite{RubinTheoriesOfLinearOrder}:
	
	\begin{theorem}[Rubin]\label{FiniteAxiomatizationTheorem}
		If $T$ is a CLO in a finite language and $S_1(T)$\footnote{Here and throughout, $S_1(T)$ refers to the set of all complete 1-types in some variable $x$ with no parameters.} is finite, then $T$ is finitely axiomatizable.
	\end{theorem}

	\subsection{Measures of Complexity}
	
	The number of nonisomorphic countable models of a theory is deceptively coarse.  One excellent refinement of this idea is the idea of Borel reduction, introduced essentially in \cite{FriedmanStanleyBC}.
	
	Given two appropriate Borel spaces $X$ and $Y$ and two equivalence relations $E$ and $F$ on $X$ and $Y$, respectively, we say $f:X\to Y$ is a \emph{Borel reduction} if $f$ is Borel and it is a reduction; that is, for all $x,x'\in X$, $xEx'$ if and only if $f(x)Ff(x')$.  We introduce the notation $(X,E)\borelleq (Y,F)$ to mean there is a Borel reduction from $(X,E)$ to $(Y,F)$.  We say $(X,E)\borellt (Y,F)$ if $(X,E)\borelleq (Y,F)$ but not conversely, and we say $(X,E)\boreleq (Y,F)$ if both $(X,E)\borelleq (Y,F)$ and conversely.
	
	This is relevant to model theory as follows.  Given a sentence $\Phi\in L_{\omega_1\omega}$, our Polish space will be $\Mod_\omega(\Phi)$, the space of $L$-structures on $\omega$ which model $\Phi$.  This has the formula topology, where our basic open sets are $U_\phi=\{M\in \Mod_\omega(\Phi):M\models\phi(\o n)\}$ where $\phi(\o x)$ is a formula and $\o n$ is a tuple from $\omega^{|\o x|}$.  The equivalence relation will be $\iso$.  Thus we will say for example that $\Phi\borelleq \Psi$, when really we mean $(\Mod_\omega(\Phi),\iso)\borelleq (\Mod_\omega(\Psi),\iso)$.  This framework is identical the one in \cite{FriedmanStanleyBC} or \cite{GaoIDST}.  Also, since we only care about the Borel sets, this is equivalent to (for example) forming a subbasis of $U_\phi$ where $\phi$ must be an atomic formula.
	
	While the space of equivalence relations as ordered by $\borelleq$ is extremely vast and complicated, we will only need a few standard examples as comparators.  First, for any $n\in\omega$, the equality relation for an $n$-element space, denoted $(n,=)$, will be relevant.  Evidently $T\boreleq (n,=)$ if and only if $T$ has exactly $n$ countable models up to isomorphism.
	
	Next, we will define $\iso_1$ as the equality relation $(\b R,=)$; if $T$ ``provably in $\ZFC$'' has continuum-many models (see \cite{FriedmanStanleyBC} for a more precise statement), then $\iso_1\borelleq T$.  Define $\iso_2$ as the set equality relation $(\b R^\omega,E)$, where $fEg$ if and only if $\{f(n):n\in\omega\}$ and $\{g(n):n\in\omega\}$ are equal as sets.  It is a theorem of Marker in \cite{MarkerNonSmallTheories} that if $S_1(T)$ is uncountable, then $\iso_2\borelleq T$.  We refer the reader to \cite{HjorthKechrisLouveau}, \cite{GaoIDST}, or \cite{FriedmanStanleyBC} for the significance of the $\iso_\alpha$ hierarchy for countable ordinals $\alpha$, as well as proofs that they are distinct.
	
	Each of the preceding examples is minimal in an important sense, although these equivalence relations form a $\borellt$-strictly increasing chain.   On the other extreme, we say $\Phi\in L_{\omega_1\omega}$ is \emph{Borel complete} if, for any $\Psi\in L_{\omega_1\omega}$ in any countable language, $\Psi\borelleq \Phi$.  It is a theorem of Friedman and Stanley in \cite{FriedmanStanleyBC} that such objects exist and are somewhat plentiful; indeed the (incomplete) theory of linear orders is Borel complete.
	
	Generalizing this to uncountable models takes a bit of a mental shift.  Classical stability theory insists that linear orders are all unstable, so have $2^\kappa$ models of size $\kappa$ for all uncountable $\kappa$, and that this is the end of the story.  For our purposes however, $\th(\b Q,<)$ has only one model (of any size), for if we take $M,N\models\th(\b Q,<)$, then \emph{in any forcing extension in which $M$ and $N$ are countable}, $M\iso N$.  Consequently, while $M$ and $N$ may be nonisomorphic, there is no logical property distinguishing the two.  The way to make this precise is by use of back-and-forth equivalence.
	
	Two $L$-structures $M$ and $N$ are said to be back-and-forth equivalent, denoted $M\equiv_{\infty\omega}N$, if there is a back-and-forth system between them.  This is equivalent to $M$ and $N$ satisfying the same sentences of $L_{\infty\omega}$, or even $L_{\lambda^+\omega}$, where $\lambda=|M|$ (see for example \cite{BarwiseScottSentences} or \cite{KeislerMT}).  This is a highly absolute notion, so $M\equiv_{\infty\omega}N$ does not become true (or false) when moving between absolute models of $\ZFC$.  Further, if $M\iso N$, then $M\equiv_{\infty\omega}N$, and this reverses in the case where $M$ and $N$ are both countable.  So $M\equiv_{\infty\omega}N$ if and only if, in some (any) forcing extension $\b V[G]$ collapsing $|M|$ and $|N|$ to $\aleph_0$, $M\iso N$.  The details of this are worked out in \cite{UlrichLaskowskiRast} and in \cite{ModelsAndGames}.
	
	Evidently structures of different cardinalities can be back-and-forth equivalent, so it makes sense to count the number of models of $\Phi$, of \emph{any} cardinality, modulo back-and-forth equivalence.  We denote this count $I_{\infty\omega}(\Phi)$; we have already shown that if $\Phi$ is $\aleph_0$-categorical, $I_{\infty\omega}(\Phi)=1$. In case there is a proper class of such models, we say $I_{\infty\omega}(\Phi)=\infty$.
	
	Following \cite{LaskowskiShelahAleph0Stable}, we can also generalize Borel reductions to uncountable cardinals.  To do so, for any infinite cardinal $\lambda$ and any $\Phi\in L_{\lambda^+\omega}$, let $\Mod_\lambda(\Phi)$ be the space of $L$-structures with universe $\lambda$ which model $\Phi$.  We make this a topological space using atomic formulas to form a subbasis, as with $\Mod_\omega(\Phi)$. A function $f:\Mod_\lambda(\Phi)\to \Mod_\lambda(\Psi)$ is said to be $\lambda$-Borel if the preimage of any subbasic open set is $\lambda$-Borel, meaning it can be formed as a usual Borel set, but with conjunctions and disjunctions of size at most $\lambda$.  Because of the presence of parameters from $\lambda$, it can easily be seen that the $\lambda$-Borel subsets of $\Mod_\lambda(\Phi)$ are precisely (infinite) Boolean combinations of subbasic open sets, so there is no incongruity with \cite{LaskowskiShelahAleph0Stable}.
	
	A $\lambda$-Borel function $f:\Mod_\lambda(\Phi)\to\Mod_\lambda(\Psi)$ is a \emph{$\lambda$-Borel reduction} when for all $M,N\in\Mod_\lambda(\Phi)$, $M\equiv_{\infty\omega}N$ if and only if $f(M)\equiv_{\infty\omega}f(N)$.  We denote the existence of such a function by saying $(\Mod_\lambda(\Phi),\equiv_{\infty\omega})\borelleq (\Mod_\lambda(\Psi),\equiv_{\infty\omega})$, often shortened to $\Phi\borelleq^\lambda \Psi$.  We say $\Phi$ is \emph{$\lambda$-Borel complete} if, for all $\Psi\in L_{\lambda^+\omega}$, $\Psi\borelleq^\lambda \Phi$.  Observe that in the case that $\lambda=\aleph_0$, we recover the original notion of Borel reductions, Borel completeness, and so on, since back-and-forth equivalence is the same as isomorphism for countable structures; thus examples exist in that case.  But actually such sentences exist for all $\lambda$:
	
	\begin{theorem}[Laskowski, Shelah]\label{SubtreesComplexTheorem}
		For any infinite cardinal $\lambda$, the class of (downward closed) subtrees of $\lambda^{<\omega}$ is $\lambda$-Borel complete.
	\end{theorem}
	
	To make this completely precise, we fix a bijection $\lambda^{<\omega}\to \lambda$ so that $\lambda$ has a tree structure on it.  Then a ``subtree of $\lambda^{<\omega}$'' is formed by expanding this structure by a unary predicate whose realizations are downward-closed with regard to the tree order, and outside of which we forget the tree order, along with some standard tricks so that the complement of the ``subtree'' is always infinite, and thus irrelevant to the back-and-forth equivalence structure.  In \cite{LaskowskiShelahAleph0Stable}, Laskowski and Shelah introduce the notion of ``$\lambda$-Borel complete for all $\lambda$'' as a kind of maximal level of complexity of a theory, and using Theorem~\ref{SubtreesComplexTheorem} as a ``test class,'' they also produce a large class of examples.  For our purposes we will need a different test class:
	
	\begin{theorem}\label{LinearOrdersBorelCompleteTheorem}
		Let $\LO$ be the sentence ``$<$ is a linear order'' in the language $\{<\}$.  Then $\LO$ is $\lambda$-Borel complete for all $\lambda$.  In particular, for all infinite $\lambda$, there are $2^\lambda$ pairwise back-and-forth inequivalent linear orders of size $\lambda$, so $I_{\infty\omega}(\LO)=\infty$.
	\end{theorem}
	\begin{proof}
		The ``in particular'' is a corollary, as follows.  Trivially there are at most $2^\lambda$ orders of size $\lambda$, up to back-and-forth equivalence (or isomorphism).  For the other direction, it is a classical result that distinct ordinals are back-and-forth inequivalent; therefore, there are at least $\lambda^+$ linear orders of size $\lambda$, indexed by the interval $[\lambda,\lambda^+)$.  So consider the language $\{E,<\}$, and the incomplete theory which states that $E$ is an equivalence relation and $<$ is a linear order on each class (but not well-defined between classes).  Then for any $X\subset [\lambda,\lambda^+)$ of size at most $\lambda$, we can make $M_X$ which has one $E$-class isomorphic to $(x,<)$ for each $x\in X$.  If $X\not=Y$ then $M_X\not\equiv_{\infty\omega}M_Y$, so there are at least $[\lambda^+]^{\leq\lambda}= 2^\lambda$ inequivalent linear orders of size $\lambda$.  This holds for all $\lambda$, proving that $I_{\infty\omega}(\LO)=\infty$.
		
		The main result is a modification of Friedman and Stanley's proof that linear orders are Borel complete, with the following paradigm.  Informally, if a Borel reduction makes sense for all infinite cardinalities and preserves size, the same construction will yield a $\lambda$-Borel reduction for all $\lambda$.
		
		Let $\lambda$ be any infinite cardinal. It follows from Theorem~\ref{SubtreesComplexTheorem} that there is a finite language $L$ where $\Mod_\lambda(L)$ -- the space of $L$-structures with universe $\lambda$ -- is $\lambda$-Borel complete.  To imitate the original proof we need a notion of a \emph{$\lambda$-dense linear order}: a structure of size $\lambda$ in the language $\{<\}\cup\{P_\alpha:\alpha\in\lambda\}$ where $<$ is a dense linear order without endpoints, the $P_i$ are disjoint unary predicates, and they are dense, codense, and exhaustive in the order.  Note that this may be weaker with some existing definitions, but this is all we will need.
		
		With the exception of exhaustion (that is, that every element is in some $P_\alpha$), this is a collection of first-order axioms.  Every finite subset of these has a countable model; just take a copy of $(\b Q,<)$ with $n$ dense, codense subsets specified.  So the whole theory has a model of size $\lambda$, and if we drop the unsorted elements, it's still a model and is a $\lambda$-dense order.  Fix some particular $\lambda$-dense linear order $I=(I,<,P_\alpha)_{\alpha\in\lambda}$.
		
		We now follow the proof of Theorem~3 from \cite{FriedmanStanleyBC} quite closely, even matching notation as much as possible.  We need to define the linear order $I_{<\omega}$ as a directed union $\bigcup_n I_n$.  We say $I_{-1}$ is a singleton.  For each $n\in\omega$, we say $I_n$ is $I_{n-1}\times (\{-\infty\}+ I)$ with the lexicographic ordering on the product and the sum. Here we identify $I_{n-1}$ with $I_{n-1}\times \{-\infty\}$ inside $I_n$.  For any $x\in I_{<\omega}$, define $\ell(x)$ as the least $n$ where $x\in I_n$.
		
		We give a labeling $f$ of $I_{<\omega}$ by $\lambda^{<\omega}$ satisfying the following conditions:
		
		\begin{itemize}
			\item If $\ell(x)=n$, then $f(x)\in \lambda^n$.
			\item If $x\in I_n$, then $f$ maps $\{x\}\times I$ \emph{onto} $\{f(x)\frown \alpha:\alpha\in\lambda\}$.
			\item For any $x\in I_n$ and any $\alpha\in\lambda$, $f^{-1}(\{f(x)\frown \alpha:\alpha\in\lambda\}$ is dense in $\{x\}\times I$.
		\end{itemize}
		
		In both of the above, as well as what follows, $\cap$ refers to concatenation of sequences.
		
		We define $f$ by induction.  If $\ell(x)=0$, $f(x)$ is the empty sequence $()\in\lambda^0$.  If $\ell(x)=n+1$, then $x=(y,i)$ for some $y\in I_n$ and some $i\in I$, and there is a unique $\alpha\in\lambda$ where $I\models P_\alpha(i)$.  So let $f(x)=f(y)\frown (\alpha)$.  Visibly this function has the desired properties, using $\lambda$-density of $I$.
		
		Next, for each $n\in\omega$, let $\TY_n$ be the set of all complete atomic $L$-types in variables $x_1,\ldots,x_n$; since $L$ is finite, so is $\TY_n$.  Let $e(0)=0$, and for each $n\in\omega$, let $e(n+1)=e(n)+|\TY_n|$. Let $\TY=\bigcup_n\TY_n$. We fix some bijection $k:\TY\to\omega$, so that if $p\in \TY_n$, then $e(n)\leq k(p)<e(n+1)$.
		
		We can now finally produce our $\lambda$-Borel reduction.  Let $A$ be an $L$-structure with universe $\lambda$, where $L$ is the fixed finite language from before. We construct a linear order $M_A$ with universe $\lambda$ in a $\lambda$-Borel way, such that for any $L$-structures $A$ and $B$ on $\lambda$, $A\equiv_{\infty\omega}B$ if and only if $M_A\equiv_{\infty\omega}M_B$.  We construct $M_A$ from $A$ by expanding $I_{<\omega}$ according to $A$.
		
		So for any $x\in I_{<\omega}$ with $\ell(x)=n$, there is a corresponding tuple $f(x)\in\lambda^n$, and this tuple has an atomic type $\otp^A(f(x))$, which has a corresponding index $k(\otp^A(f(x)))$.  So let $J_x$ be the linear order $\b Q+ 2+k(\otp^A(f(x)))+\b Q$; this is a dense piece, followed by a long enough finite piece not to disappear but which uniquely captures the type of $f(x)$, followed by a dense piece to separate this information from others.  So let $M_A$ be the sum $\sum_x J_x$.  The map $A\mapsto M_A$ can easily be made a $\lambda$-Borel function from $\Mod_\lambda(L)$ to $\Mod_\lambda(\LO)$; the detail to check is that each $J_x$ is countable and $I_{<\omega}$ is a fixed set of size $\lambda$, so $|\sum_x J_x|$ can be put into (more or less) canonical bijection with $\lambda$.
		
		To show it is a reduction, let $\b V[G]$ be a forcing extension in which $\lambda$ is countable (e.g. a Levy collapse of $\lambda^+$ to $\omega_1$ will do).  Observe that $A\equiv_{\infty\omega}B$ if and only if $A\iso B$ in $\b V[G]$, and likewise with $M_A$ and $M_B$.  So pass to $\b V[G]$.  Once there, observe that $I$ is isomorphic to any $\aleph_0$-dense partition of $(\b Q,<)$, and $A$ and $B$ are (up to isomorphism) just elements of $\Mod_\omega(L)$.  Therefore, this collapses to the exact construction showing $\Mod_\omega(L)\borelleq \Mod_\omega(\LO)$ from \cite{FriedmanStanleyBC}, so $A\iso B$ (in $\b V[G]$) if and only if $M_A\iso M_B$ (in $\b V[G]$).  This completes the proof.
	\end{proof}
	
	It is not clear that there should be a strong connection between $I_{\infty\omega}(\Phi)$ and the complexity of isomorphism for countable models of $\Phi$.  For example, in \cite{UlrichLaskowskiRast}, the authors demonstrate examples of a Borel complete theory $T$ with $I_{\infty\omega}(T)=\beth_2$.  It is therefore remarkable that \emph{in all of the cases in this paper}, they line up exactly.  That is, if $T$ is a CLO and $T\boreleq \iso_1$ (informally, real numbers), then $I_{\infty\omega}(T)=\beth_1$.  If $T\boreleq \iso_2$ (informally, countable sets of reals), then $I_{\infty\omega}(T)=\beth_2$.  And if $T$ is Borel complete, then $T$ is $\lambda$-Borel complete for all infinite $\lambda$, so in particular $I_{\infty\omega}(T)=\infty$.

	\section{Self-Additive Linear Orders}\label{SelfAdditiveSection}
	
	The crux of the proof is a clever definition due to Rubin -- the notion of self-additivity.  We summarize the basic properties of self-additive orders, following from Theorem~3.2 and Lemma~3.4 in \cite{RubinTheoriesOfLinearOrder}:
	
	\begin{defthm}
		Let $A$ be a CLO with more than one point.  The following properties are equivalent:
		
		\begin{itemize}
			\item The only $\emptyset$-definable convex subsets of $A$ are $\emptyset$ and $A$.
			\item If $(J,\leq )$ is a linear order and $A_j\equiv A$ for all $j\in (J,\leq)$, then each embedding $A_j\to \sum_{j\in J}A_j$ is elementary.  In particular, $\sum_j A_j \equiv A$.
		\end{itemize}
		
		If $A$ satisfies any of these properties, call $A$ \emph{self-additive}.
	\end{defthm}
	
	For example, each of $(\b Z,\leq)$, $(\b Q,<)$ are self-additive.  Indeed $(\b R,<,\b Q)$, which is the real order with a predicate marking whether a number is rational, is self-additive. However, neither $(\b N,<)$ nor $(\b Z+1+\b Z,<)$ is.  Self-additive structures are extremely useful for us because we can easily construct models using the sum operation.  They also have another important property, namely, a nice condensation on the models.
	
	\begin{definition}
		Let $A$ be a self-additive CLO, and let $a$ and $b$ be from $A$.  Say $a\sim b$ if there is a formula $\phi(x,a)$ where $\phi(A,a)$ is convex, bounded, and contains both $a$ and $b$.
	\end{definition}
	
	Note that we consider a set bounded if there are elements strictly above and strictly below the entire set.  Since self-additive orders cannot have first or last elements, this is equivalent to any other reasonable definition.  It is a theorem of Rubin in \cite{RubinTheoriesOfLinearOrder} that $\sim$ is a equivalence relation on $I$ with convex classes, although this is spelled out more plainly in Theorem~13.99 of \cite{RosensteinLinearOrderings}.  The reader is cautioned that neither symmetry nor transitivity is simple to verify, but we do not reproduce the details here.
	
	The following is the main way we will show complexity of CLOs:
	
	\begin{lemma}\label{ComplexityDriverLemma}
		Suppose $A$ is a self-additive CLO, $T=\th(A)$, and $p\in S_1(T)$ such that there is only one $\sim$-class in $A$ containing a realization of $p$.  Then $T$ is $\lambda$-Borel complete for all $\lambda$.
	\end{lemma}
	\begin{proof}
		Let $A_0\prec A$ be countable and contain a realization of $p$.  If $a,b\in A_0$, then $a\sim b$ in $A_0$ if and only if $a\sim b$ in $A$, so $A_0$ still satisfies the hypotheses of the theorem.  This is to say, we may assume $A$ is countable, and in fact that $A$ has universe $\omega$.  Fix an infinite cardinal $\lambda$ and a canonical bijection $\lambda\times\omega\to\lambda$.  We aim to show that $(\LO_\lambda,\equiv_{\infty\omega})\borelleq^\lambda(\Mod_\lambda(T),\equiv_{\infty\omega})$; by Theorem~\ref{LinearOrdersBorelCompleteTheorem}, this shows that $T$ is $\lambda$-Borel complete.
		
		For any linear order $(I,<)$ with universe $\lambda$, let $A_I=I\times A=\sum_{i\in I}A$, which has universe $\lambda\times\omega$ since $A$ has universe $\omega$. Under the bijection we may assume $A_I$ has universe $\lambda$.  Let $(J,<)$ be another linear order with universe $\lambda$, and let $A_J$ be formed in the same manner as $A_I$.  Let $\b V[G]$ be a forcing extension in which $\lambda$ is countable, so that $I$, $J$, $A_I$, and $A_J$ are all countable in $\b V[G]$.  Then $(I,<)\equiv_{\infty\omega}(J,<)$ if and only if $(I,<)\iso (J,<)$ in $\b V[G]$, and likewise with $A_I$ and $A_J$.  This is all to say we may work solely in the countable case, with isomorphism.
		
		Observe that for any $a\in A_i$, if $\phi(x,a)$ is convex, bounded, and contains $a$ in $A_I$, then it is bounded in $A_i$ as well, and these bounds still apply in $A_I$ -- all this follows from self-additivity, which enforces $A_i\prec A_I$.  Therefore, if $a,b\in A_I$ and $a\sim b$, then $a$ and $b$ are in the same summand.
		
		Now clearly if $I\iso J$, then $A_I\iso A_J$.  On the other hand, consider the set of $\sim$-classes $E_I=\{a\class :A_I\models p(a)\}$.  These are naturally ordered by $<$, and if $a\sim b$ in $A_I$, then they come from the same summand $A_i$, and are equivalent in $A_I$ if and only if they're equivalent in that $A_i$.  Since each $A_i$ contains exactly one $\sim$-class containing a realization of $p$, $E_I$ has order type $(I,<)$.  Clearly if $A_I\iso A_J$, then $(E_I,<)\iso (E_J,<)$, so $I\iso J$, completing the proof.
	\end{proof}
	
	For example, this shows that $\th(\b Z,<)$ is $\lambda$-Borel complete for all $\lambda$, since $(\b Z,<)$ has a unique $\sim$-class. But it can be used much more generally than that.  We borrow Lemma~6.1 of \cite{RubinTheoriesOfLinearOrder}:
	
	\begin{lemma}[Rubin]\label{Rubin 6.1}
		Let $A\equiv B$ be self-additive, $T=\th(A)$.  Let $b\in B$ be arbitrary.  Then the canonical embedding from $A+(b\class)+A$ to $A+B+A$ is elementary.
	\end{lemma}
	
	\begin{lemma}\label{SelfAdditiveManyTypesBCLemma}
		Let $T$ be a theory of a self-additive CLO such that $S_1(T)$ is infinite.  Then $T$ is $\lambda$-Borel complete for all $\lambda$.
	\end{lemma}
	\begin{proof}
		Let $p\in S_1(T)$ be nonisolated.  Let $A,B\models T$ be countable such that $A$ omits $p$ and $B$ realizes $p$ at $b$.  Let $B_0=b\class$ as computed in $B$, and let $C=A+B_0+A$. By Lemma~\ref{Rubin 6.1}, $C\prec A+B+A$ is elementary.  Since $A,B\models T$ and $T$ is self-additive, $A+B+A\models T$, so $C\models T$.  Also, both embeddings $A\to A+B+A$ are elementary, so in particular, no element of $A$ is $\sim$-equivalent to any element of $B_0$.  Similarly, since $B\prec A+B+A$ and every element of $B_0$ is $\sim$-equivalent in $B$, they are still $\sim$-equivalent in $A+B+A$, and thus in $C$.  Finally, $c\in C$ realizes the same type it does in $A+B+A$, and thus $C$ contains a unique $\sim$-class containing a realization of $p$.  So $C$ and $T$ satisfy the hypotheses of Lemma~\ref{ComplexityDriverLemma}, so $T$ is $\lambda$-Borel complete for all $\lambda$.
	\end{proof}
	
	Of course we cannot say the same when $S_1(T)$ is finite -- $(\b Q,<)$ is $\aleph_0$-categorical, and thus as far from Borel complete as one could be.  Our aim is to show that these are the only two cases which can occur, but we need to move slightly beyond the self-additive case to do so.  We borrow Lemma~5.4 of \cite{RubinTheoriesOfLinearOrder}:
	
	\begin{lemma}[Rubin]\label{Rubin 5.4}
		Let $T$ be a CLO with $S_1(T)$ finite.  If $A\models T$ and $a\in A$, let $T_a=\th(a\class)$.  Then $|S_1(T_a)|\leq |S_1(T)|$.  Also, one of the following alternatives holds:
		
		\begin{enumerate}
			\item For every $a\in A$, $(a\class)\prec A$.
			\item For every $a\in A$, the set $(a\class)$ is definable over $a$.  There is no first or last element in the quotient order $A\class$, and if $a\class<b\class$ and $p\in S_1(T)$, there is a $c$ realizing $p$ where $a\class<c\class<b\class$.
		\end{enumerate}
	\end{lemma}
	
	With this lemma in hand, we can finish our work with self-additive structures:
	
	\begin{lemma}\label{FinitelyManyTypesLemma}
		Let $T$ be a theory of a self-additive CLO such that $S_1(T)$ is finite.  Then either $T$ is $\aleph_0$-categorical or $T$ is $\lambda$-Borel complete for all $\lambda$.
	\end{lemma}
	\begin{proof}
		Let $L$ be the underlying language of $T$, and let $n=|S_1(T)|$; we argue by induction on $n$, beginning with $n=1$.
		
		If $n=1$ then every element has the same type, and thus the same color, so essentially $T$ is the theory of a linear order.  If $T$ says the order has a first element, every element is first, so $T$ is the theory of a singleton, so is $\aleph_0$-categorical (and indeed totally categorical); the same happens with a last element.  Assume this does not happen, so there is no first or last element.  If some element has a unique successor, they all do, and their successors have predecessors, so everything does.  This is known to be an axiomatization of $(\b Z,<)$ (where the colors in $L$ are either uniformly true or uniformly false). This has a unique $\sim$-class, so is $\lambda$-Borel complete for all $\lambda$ by Lemma~\ref{ComplexityDriverLemma}.  Now assume these do not happen, so no element has an immediate successor or predecessor and there are no maximal or minimal elements.  This is known to be an axiomatization of $(\b Q,<)$, which is $\aleph_0$-categorical, completing the proof of the base case.
		
		We move on to the step, where $n\geq 2$. There are several cases.
		
		{\bf Case:} $T$ is not self-additive.
		
		Let $\phi(x)$ be a formula with parameters such that (according to $T$), the realizations of $\phi$ form a nonempty proper initial segment of the model.  Let $A\models T$, let $T_1=\th(\phi(A))$, and let $T_2=\th(\lnot\phi(A))$.  Note that $T_1$ and $T_2$ depend only on $T$, not on $A$.  If both are $\aleph_0$-categorical, so is $T$, by Theorem~\ref{Aleph0CatCharacterizationThm}.  On the other hand, given any model $A\models T$ and $B\models T_1$, we can construct a structure $A_B$ where we replace $\phi(A)$ with $B$.  By Proposition~\ref{DefinableSumProp1}, $A_B\models T$ and $\phi(A_B)=B$.  Evidently this gives a $\lambda$-Borel reduction $\Mod(T_1)\borelleq^\lambda \Mod(T)$, so if $T_1$ is $\lambda$-Borel complete, so is $T$.  The same goes for $T_2$.  Since $|S_1(T)|=|S_1(T_1)|+|S_1(T_2)|$, the inductive hypothesis applies to both $T_i$.  Thus, either both $T_i$ are $\aleph_0$-categorical or one of them is $\lambda$-Borel complete for all $\lambda$. So the lemma holds for $T$.
		
		{\bf Case:} $T$ is self-additive and case (1) of Lemma~\ref{Rubin 5.4} applies.
		
		Let $A\models T$, let $a\in \m A$ be arbitrary, and let $B=a\class$.  Then $B\prec A$, so $B\models T$, so $T$ has a model with a single $\sim$-class.  Then Lemma~\ref{ComplexityDriverLemma} applies, so $T$ is $\lambda$-Borel complete for all $\lambda$.
		
		{\bf Case:} $T$ is self-additive and case (2) of Lemma~\ref{Rubin 5.4} applies.
		
		Let $A\models T$ be arbitrary.  Each $\sim$-class is an $L$-structure on its own, and the theory of $a\class$ is determined by $\tp(a)$.  So there are $k\leq n$ $\sim$-classes up to elementary equivalence; enumerate their theories as $T_1,\ldots,T_k$.  To simplify notation, add $k$ unary predicates $U_1,\ldots, U_k$ to the language, and expand $T$ to the new language by saying $U_i(a)$ holds if and only if $a\class\models T_i$.  Since this is a definable expansion, this does not change the size of the type space, and $T$ satisfies the lemma if and only if its expansion does.
		
		Then $T$ states precisely that each maximal convex piece of $U_i$ is a model of $T_i$, and between any two ``convex pieces'' and for any $i\leq k$, there is a model of $T_i$ as a maximal convex piece of $U_i$.  It states that the $U_i$ are disjoint and exhaustive.  Fix particular countable models $A_i\models T_i$, and for any $M$ fitting the preceding description, form the $L$-structure $A$ by replacing each maximal convex piece of any $U_i$ with the $L$-structure $A_i$.  This can be done, and by Proposition~\ref{BasicSumProp}, $M\equiv A_M$.  However, given $M$ and $N$ fitting the description, $A_M\iso A_N$ by Proposition~\ref{GeneralShuffleProposition}, so the preceding description is a complete theory, so must completely axiomatize $T$.
		
		Next, see that $\sum_i |S_1(T_i)|\leq |S_1(T)|$.  For if $a$ and $b$ come from different $U_i$, they have different types in $T$.  And since the $\sim$-class of an element is formula-definable with that element, if $a$ and $b$ come from the same $U_i$ but have different types in that structure, they have different types in $T$.  Therefore, if $k\geq 2$, then $|S_1(T_i)|<|S_1(T)|$ for all $i$, so the inductive hypothesis applies to each of them.  If each is $\aleph_0$-categorical, so is $T$ by Proposition~\ref{Aleph0CatCharacterizationThm} -- $T$ is the shuffle of the $T_i$. If some $T_i$ is $\lambda$-Borel complete, then so is $T$, as follows.  Let $M\models T$ be countable and fixed.  For any $A\models T_i$ of size $\lambda$, let $M_A$ be formed by replacing each maximal convex model of $T_i$ with $A$.  Then $|M_A|=\lambda$, and given a $B\models T_i$ also of size $\lambda$, form $M_B$ in the same way. If $M_A\iso M_B$, this isomorphism preserves maximal convex pieces of $U_i$, so $A\iso B$.  With this in mind, let $\b V[G]$ collapse $\lambda$, so $A\equiv_{\infty\omega} B$ if and only if $A\iso B$ in $\b V[G]$, if and only if $M_A\iso M_B$ in $\b V[G]$, if and only if $M_A\equiv_{\infty\omega}M_B$.  This shows $\Mod_\lambda(T_i)\borelleq^\lambda \Mod_\lambda(T)$, so $T$ is also $\lambda$-Borel complete.
		
		The only remaining case is when $k=1$, so $T$ is a shuffle of $T_1$.  If $T_1$ is self-additive, then each $\sim$-class of any $A\models T$ is an elementary substructure of $A$, so $T_1=T$ and $T$ admits a model with a single $\sim$-class.  However, since the $\sim$-class of any element is definable, $T$ would then imply that every model has only one $\sim$-class, contradicting what we already know about $T$. So $T_1$ is not self-additive.  Then a previous case applies to $T_1$, so $T_1$ is either $\aleph_0$-categorical or is $\lambda$-Borel complete for all $\lambda$.  In either case, $T$ follows $T_1$ by the logic in the previous paragraph, completing the proof.
	\end{proof}
	
	While we do not care about orders with finitely many types for themselves, we do recover the following theorem which is crucial to us:
	
	\begin{theorem}\label{SADichotomyTheorem}
		Let $T$ be self-additive.  Then either $T$ is $\aleph_0$-categorical or is $\lambda$-Borel complete for all $\lambda$.
	\end{theorem}

	\section{The General Proof}\label{IntervalTypesSection}
	
	Finally we consider the general case, where we break up arbitrary CLOs into what are essentially self-additive pieces.  The crucial definition is the following:
	
	\begin{definition}
		Let $T$ be a CLO. A convex formula $\phi(x)$, always in one variable, is a formula whose set of realizations in a model of $T$ is always a convex set. A \emph{convex type} $\Phi(x)$ is a maximal consistent collection of convex formulas over $\emptyset$.  The space $IT(T)$ is the set of all convex types.
	\end{definition}
	
	We may give $IT(T)$ the usual formula topology, wherein it is compact, Hausdorff, second countable, and totally disconnected as usual.  However, convex types are naturally ordered by $<$ as follows: say $\Phi<\Psi$ if there are formulas $\phi\in \Phi$ and $\psi\in \Psi$ where every realization of $\phi$ is strictly below every realization of $\Psi$ (according to $T$).  It is immediate that if $\Phi\not=\Psi$, then either $\Phi<\Psi$ or $\Psi<\Phi$, and not both.  This induces the same topology as before.
	
	For our purposes, say an $L$-structure $A$ is \emph{sufficiently saturated} if it is $\aleph_0$-saturated, and if $a,b\in A$ realize the same type, there is an automorphism of $A$ taking $a$ to $b$; we will never need a larger monster model than this.  Every complete theory admits such a model, although they need not be countable unless the theory is small.  Sufficiently saturated CLOs are ``locally self-additive:''
	
	\begin{lemma}
		Let $T$ be a CLO, and let $\m S\models T$ be sufficiently saturated.  Then for all $\Phi\in \IT(T)$, the set $\Phi(\m S)$ of realizations of $\Phi$ in $\m S$ is either a singleton or self-additive as an $L$-structure.
	\end{lemma}
	\begin{proof}
		Suppose not.  Then there is a formula $\phi(x)$ whose realizations are initial in $\Phi(\m S)$, and there are points $a,b\in\Phi(\m S)$ where $\Phi(\m S)\models \phi(a)\land\lnot\phi(b)$.  Let $p=\tp(b)$ as formed in the whole of $\m S$.  Then the type $p(x)\cup\{x<a\}$ must be consistent; otherwise there would be some $\psi(x)\in p(x)$ where $a$ lies strictly below the convex definable set $\exists y \left(y\leq x\land \psi(y)\right)$ which includes $b$, and therefore $a$ and $b$ realize different convex types.  By $\aleph_0$-saturation of $\m S$, there is $c\in\m S$ which realizes $p$ and has $c<a$.  Clearly $\Phi(\m S)\models\phi(c)$.  But then there is an automorphism $\sigma$ of $\m S$ where $\sigma(a)=c$.  But $\sigma$ preserves $\Phi$, so is an automorphism of $\Phi(\m S)$ which takes $b$ to $c$.  Automorphisms are elementary maps, so $\Phi(\m S)\models\lnot\phi(c)$, a contradiction.
	\end{proof}
	
	The \emph{theory} of $\Phi(\m S)$ turns out to depend only on $T$, not on the choice of sufficiently saturated model:
	
	\begin{lemma}
		Let $\m S_1$ and $\m S_2$ be sufficiently saturated models of a CLO $T$.  For any $\Phi\in \IT(T)$, the $L$-structures $\Phi(\m S_1)$ and $\Phi(\m S_2)$ are back-and-forth equivalent, and thus elementarily equivalent.
	\end{lemma}
	\begin{proof}
		Our claim is that if $\o a\in \Phi(\m S_1)^n$ and $\o b\in\Phi(\m S_2)^n$, then $(\m S_1,\o a)\equiv (\m S_2,\o b)$ as structures with specified constants, and if $a\in\Phi(\m S_1)$ is arbitrary, there is a $b\in\Phi(\m S_2)$ where $(\m S_1,\o aa)\equiv (\m S_2,\o bb)$ as structures with specified constants.  This implies that $\o aa$ and $\o bb$ have the same atomic type in the substructures, so together with the opposite (which follows from symmetry) gives the result.
		
		So fix such tuples $\o a$, $\o b$, and $a$.  Then $p(\o x,x)=\tp(\o a,a)$ (evaluated in $\m S_1$) is realized and thus consistent with $T$.  Therefore it is realized in $\m S_2$ by some pair $\o c,c$.  But then $\tp(\o c)=\tp(\o a)=\tp(\o b)$, so there is an automorphism $\sigma$ of $\m S_2$ taking $\o c$ to $\o b$; let $b=\sigma(c)$.  Clearly $\tp(\o aa)=\tp(\o cc)=\tp(\o bb)$, and $b\in\Phi(\m S_2)$ since $c$ is, so $b$ satisfies the conditions and proves the result.
	\end{proof}
	
	By the Lemma, $\Phi(\m S_1)\equiv \Phi(\m S_2)$, so we may define the notation $T_\Phi$ to be the complete $L$-theory of $\Phi(\m S)$ for any sufficiently saturated $\m S\models T$.  We can now declare our fundamental dichotomy:
	
	\begin{definition}
		Let $T$ be a CLO.  Say $T$ is \emph{locally simple} if for all $\Phi\in \IT(T)$, $T_\Phi$ is $\aleph_0$-categorical.  Otherwise say $T$ is \emph{locally nonsimple}.
	\end{definition}
	
	\begin{theorem}
		If $T$ is a locally nonsimple CLO, $T$ is $\lambda$-Borel complete for all $\lambda$.
	\end{theorem}
	\begin{proof}
		Let $\m S\models T$ be sufficiently saturated, and let $\Phi\in \IT(T)$ be such that $T_\Phi$ is not $\aleph_0$-categorical.  Let $A\prec\m S$ be countable such that $\Phi(A)\prec \Phi(\m S)$.  Then $\Phi(A)\models T_\Phi$, which is a self-additive CLO which is $\lambda$-Borel complete for all $\lambda$ by Theorem~\ref{SADichotomyTheorem}.  Furthermore, $T_\Phi\borelleq^\lambda T$ as follows.  If $B\in\Mod_\lambda(T_\Phi)$, construct $A_B$ by replacing $\Phi(A)$ by $B$; then $A\equiv A_B$ and $\Phi(A_B)=B$ by Proposition~\ref{DefinableSumProp1}, so in particular $A_B\in\Mod_\lambda(T)$.  Clearly $B\iso B'$ if and only if $A_B\iso A_{B'}$, so by using that fact in some $\b V[G]$ which collapses $\lambda$, $B\equiv_{\infty\omega} B'$ if and only if $A_B\equiv_{\infty\omega}A_{B'}$, completing the proof.
	\end{proof}
	
	Since the global behavior of a CLO $T$ is determined essentially by the structure of $\IT(T)$, if there is also local simplicity, there is little opportunity for complexity.  Thus, locally simple CLOs turn out to admit a nice characterization. We borrow Lemma~2.7(1) of \cite{RubinTheoriesOfLinearOrder}: if $B$ is a CLO, $A\prec B$, and $C$ is the convex hull of $A$ in $B$ (that is, the minimal convex subset of $B$ which contains $A$), then $A\prec C$.  The following is the core lemma for understanding this case:
	
	\begin{lemma}\label{LocallySimpleDriverLemma}
		Let $T$ be locally simple and $\Phi\in \IT(T)$. Then there is a (minimal) natural number $n_\Phi$ and a set $\{T_\Phi^i:1\leq i\leq n_\Phi\}$ of distinct $\aleph_0$-categorical $L$-theories where for all $A\models T$, there is an $i$ where $\Phi(A)\models T_\Phi^i$.
		
		Against model-theoretic convention, we include the ``theory of the empty set'' in the list, where we say $\emptyset\models \forall x(x\not=x)$, in case $A$ omits $\Phi$.
		
		Further, $n_\Phi=1$ if and only if $\Phi$ is isolated as a type in $\IT(T)$.
	\end{lemma}
	\begin{proof}
		Let $A\models T$.  Then $\Phi(A)\subset \Phi(\m S)$ for some sufficiently saturated $\m S\models T$ where $A\prec\m S$.  Let $C$ be the convex hull of $\Phi(A)$ in $\Phi(\m S)$; then $A\prec C$, so they have the same theory.  Also, $\Phi(C)$ is a convex subset of an $\aleph_0$-categorical CLO, so is $\aleph_0$-categorical by Proposition~\ref{Aleph0CatGeneralFactsProp}.  So $\Phi(A)$ is $\aleph_0$-categorical.  Also by Proposition~\ref{Aleph0CatGeneralFactsProp}, there are only finitely many pairwise inequivalent convex subsets of $\Phi(\m S)$, and this bound depends only on $T_\Phi$.  So the main text of the lemma is proven.
		
		If $\Phi$ is nonisolated, there is a model omitting it and another realizing it, so $n_\Phi\geq 2$.  If $\Phi$ is isolated by some formula $\phi$, then for every sentence $\sigma$ of $L$, the sentence ``$\sigma$ holds on $\Phi$'' is equivalent to ``the relativization of $\sigma$ to $\phi$ is true,'' which is a single $L$-sentence and thus decided by $T$.  So $n_\Phi=1$.
	\end{proof}
	
	This allows us to give a simple characterization of back-and-forth equivalence for locally simple CLOs:
	
	\begin{lemma}\label{LocallySimpleCharacterizationLemma}
		Let $T$ be a locally simple CLO, and $A,B\models T$.  The following are equivalent:
		
		\begin{enumerate}
			\item $A\equiv_{\infty\omega} B$
			\item For all $\Phi\in \IT(T)$, $\Phi(A)\equiv_{\infty\omega}\Phi(B)$
			\item For all $\Phi\in \IT(T)$, $\Phi(A)\equiv \Phi(B)$.
		\end{enumerate}
		
		The equivalence of (1) and (2) does not require local simplicity.  If $A$ and $B$ are countable, (1) is equivalent to $A\iso B$, and likewise with (2).
	\end{lemma}
	\begin{proof}
		Assuming (1), we can get (2) by playing the back-and-forth game within any particular $\Phi$; we can reverse this step by patching the various solutions to the $\Phi$ together, and making sure we always play within a particular $\Phi$.  Assuming (2) we get (3) immediately, since elementary equivalence is just $\equiv_{\omega\omega}$.  The nontrivial step is to show that (3) implies (2), which follows from Lemma~\ref{LocallySimpleDriverLemma}.  For if $\Phi(A)\models T_\Phi^i$ for some $i$, then $\Phi(B)\models T_\Phi^i$, and since $T_\Phi^i$ is $\aleph_0$-categorical, all of its models are back-and-forth equivalent.  The equivalence of back-and-forth equivalence with isomorphism when the structures are countable is standard and follows from Zorn's lemma.
	\end{proof}
	
	It only remains to give an exhaustive list of the behaviors that a CLO can exhibit, both in terms of $I_{\infty\omega}(T)$ and of the Borel complexity of $\iso_T$.  We will use the following lemma to construct many models, to the extent that the type space allows it.  Approximately, we would like to choose to omit or realize whatever types we like, but if we omit too many, we don't have enough content left over to have a model.  This turns out to be the only obstruction:
	
	\begin{lemma}\label{LocallySimpleManyModelsLemma}
		Suppose $A\subset C\subset B$, that $A,B\models T$, and that $A\prec B$.  Suppose also that there is a collection $K\subset IT(T)$ where $C=B\setminus_{\Phi\in K}\Phi(B)$.  Then $C\prec B$ as well.
	\end{lemma}
	\begin{proof}
		Let $\phi(\o x,y)$ be a formula, $\o c$ a tuple from $C$, and $b$ an element from $B$ where $B\models\phi(\o c, b)$.  It's enough to show there is a $c\in C$ where $B\models\phi(\o c,c)$.  By the particular construction of $C$, either $b\in C$ or there is a convex formula $\psi(y)$ where $B\models\psi(b)$, and where no element of $\o c$ realizes $\psi$.  By Lemma~\ref{RelativizationLemma}, there is a formula $\phi^{\#}(y)$ where for all $b'$ realizing $\psi$, $\psi(B)\models\phi^{\#}(b')$ if and only if $B\models\phi(\o c,b')$.  Since $\psi$ is itself definable, there is a formula $\phi^*(y)$ where $B\models \forall y\left(\phi^*(y)\leftrightarrow \phi(\o c,y)\right)$.  Of course $B\models\phi^*(b)$, and since $A\prec B$, there is an $a\in A$ where $B\models\phi^*(a)$.  Since $A\subset C$, this $a$ is the element we were looking for, which completes the proof.
	\end{proof}
	
	We can now give individual cases:
	
	\begin{proposition}\label{LocallySimpleOneProp}
		If $T$ is locally simple and $\IT(T)$ is finite, $T$ is $\aleph_0$-categorical and $I_{\infty\omega}(T)=1$.
	\end{proposition}
	\begin{proof}
		Since $\IT(T)$ is finite, every $\Phi\in \IT(T)$ is isolated.  Thus $n_\Phi=1$ for all $\Phi$, so every $A,B\models T$ are back-and-forth equivalent by Lemma~\ref{LocallySimpleCharacterizationLemma}.  If $A$ and $B$ are also countable, they are isomorphic as well.
	\end{proof}
	
	\begin{proposition}\label{LocallySimpleFiniteProp}
		If $T$ is locally simple and $\IT(T)$ is infinite but with only finitely many nonisolated types, there is a natural number $n\geq 3$ where $\iso_T$ is Borel equivalent to $(n,=)$ and $I_{\infty\omega}(T)=n$.
	\end{proposition}
	\begin{proof}
		Let $\Phi_1,\ldots,\Phi_k\in\IT(T)$ be the nonisolated convex types, and let $m$ be the product of the $n_{\Phi_i}$ as defined in Lemma~\ref{LocallySimpleDriverLemma}.  For any $A\models T$, let $t_A$ be $(\th(\Phi_i(A)):i\leq k)$.  If $A,B\models T$, then $A\equiv_{\infty\omega}B$ if and only if $t_A=t_B$.  Further, there are at most $m$ possible sequences $t_A$, so $T$ has at most $m$ countable models up to isomorphism; call the exact count $n$.  That $n\geq 2$ comes the fact that some type is nonisolated; that $n\geq 3$ comes from the fact that $T$ is a complete first-order theory.
		
		Clearly $I_{\infty\omega}(T)\geq n$.  That $I_{\infty\omega}(T)\leq n$ comes as follows; if $A\models T$ is arbitrary, let $A_0\prec A$ be countable and have $\Phi(A_0)\prec \Phi(A)$ for all $\Phi\in \IT(T)$.  Then $A\equiv_{\infty\omega}A_0$.  And for any two models $A,B\models T$, $A\equiv_{\infty\omega} B$ if and only if $A_0\equiv_{\infty\omega}B_0$, if and only if $A_0\iso B_0$ (see \cite{MarkerMT} for this equivalence).  Since there are $n$ isomorphism types of countable models of $T$, $I_{\infty\omega}(T)\leq n$, completing the proof.
	\end{proof}
	
	\begin{proposition}\label{LocallySimpleCountableProp}
		If $T$ is locally simple and $\IT(T)$ is countable but with infinitely many nonisolated types, then $\iso_T$ is Borel equivalent to $\iso_1$ and $I_{\infty\omega}(T)=\beth_1$.
	\end{proposition}
	\begin{proof}
		We first show $\iso_T\borelleq \iso_1$ by showing a Borel reduction from $\Mod_\omega(T)$ to $(\omega^\omega,=)$; note this is Borel equivalent to $\iso_1$ (see \cite{KechrisDST}).  For each $\Phi$, fix an indexing of $\{T_\Phi^i:1\leq i\leq n_\Phi\}$.  Also fix an indexing $\{\Phi_n:n\in\omega\}$ of $\IT(T)$.  Then for any model $M\models T$, let $s_M\in\omega^\omega$ take $n\in\omega$ to the unique $i$ where $\Phi_n(A)\models T_\Phi^i$.  Certainly for any $M,N\models T$, $M\equiv_{\infty\omega} N$ if and only if $s_M=s_N$, so $\iso_T\borelleq \iso_1$.  Since this construction makes sense for any models of $T$, this also shows $I_{\infty\omega}(T)\leq \beth_1$.
		
		For the other direction, we show $\iso_1\borelleq \iso_T$ by giving a Borel reduction from $(2^\omega,=)$ to $\Mod_\omega(T)$.  Let $\{\Phi_n:n\in\omega\}$ be an enumeration of the nonisolated types in $\IT(T)$.  Let $A\models T$ be some model omitting every $\Phi_n$, and let $B\succ A$ be countable and realize every $\Phi_n$.  For $\eta\in 2^\omega$, let $C_\eta$ omit $\Phi_n\in\IT(T)$ if and only if $\eta(n)=0$. This is done by use of Lemma~\ref{LocallySimpleManyModelsLemma}, so that $C_\eta$ is just the elements of $B$ which are not in $\bigcup \{\Phi(B):\eta(n)=0\}$.  Certainly this can be made Borel and $C_\eta\iso C_\nu$ if and only if $\eta=\nu$.  Since all the $C_\eta$ are countable, this shows $\iso_1\borelleq \iso_T$.  Since these models is countable and pairwise nonisomorphic, they are also pairwise back-and-forth inequivalent.  So $I_{\infty\omega}(T)\geq \beth_1$, completing the proof.
	\end{proof}
	
	\begin{proposition}\label{LocallySimpleUncountableProp}
		If $T$ is locally simple and $\IT(T)$ is uncountable, then $\iso_T$ is Borel equivalent to $\iso_2$ and $I_{\infty\omega}=\beth_2$.
	\end{proposition}
	\begin{proof}
		We first show that $\iso_T\borelleq \iso_2$ by showing a Borel reduction from $\Mod_\omega(T)$ to $((X)^\omega,E)$, where $X$ is the set of all possible $L$-theories and two functions are equivalent if and only if their images are equal as sets; $((X)^\omega,E)$ is Borel equivalent to $\iso_2$ because $X$ is a Polish space (see \cite{KechrisDST}).  Since $X$ is a standard Borel space, $(X,E)\boreleq \iso_2$.  So let $M\in\Mod_\omega(T)$, and for each $n\in\omega$, let $\Phi^M_n$ be the convex type of $n$ in $M$.  Then let $T^M_n$ be the theory of $\Phi^M_n(M)$, and define our function by $M\mapsto (T^M_n:n\in\omega)$.  By Lemma~\ref{LocallySimpleCharacterizationLemma}, countable models $M,N\models T$ have $M\iso N$ if and only if they realize the same convex types (necessarily a countable set), and for each realized type $\Phi$, $\Phi(M)\equiv \Phi(N)$.  This is equivalent to the \emph{sets} $\{T^M_n:n\in\omega\}$ and $\{T^N_n:n\in\omega\}$ being equal.
		
		The back-and-forth version of this argument is less delicate. Two models $M,N\models T$ (of any size) are back-and-forth equivalent if and only if, for all $\Phi\in \IT(T)$, $\Phi(M)\equiv \Phi(N)$.  Since $\IT(T)$ is uncountable, $|\IT(T)|=\beth_1$, so $I_{\infty\omega}(T)\leq \omega^{\beth_1}=\beth_2$.
		
		For the reverse, we again use Lemma~\ref{LocallySimpleManyModelsLemma}.  Fix a countable model $M\models T$ and some model $\m S\models T$ realizing every convex type.  Let $X$ be the set of convex  types \emph{omitted} by $M$; since $\IT(T)$ is uncountable and $M$ is countable, $X$ is an uncountable standard Borel space using the usual topology. For any set $K\subset \IT(T)$, let $M_K$ be $\m S\setminus \bigcup_{\Phi\in X\setminus K}\Phi(\m S)$, so that $M_K$ realizes only the types in $K$.  If $K_1\not=K_2$, $M_{K_1}$ and $M_{K_2}$ realize different types, so are pairwise inequivalent.  Thus $I_{\infty\omega}(T)\geq\beth_2$.
		
		For the countable version of this argument, we need to be slightly more careful.  We restrict ourselves to countable $K$, so that $M_K$ realizes only countably many types. We also need $\Phi(\m S)$ to be countable for each $\m S$, which can be guaranteed by simply replacing each $\Phi(\m S)$ with a countable elementary substructure. But then we have a Borel function from $(X^\omega,E)$ to $(\Mod_\omega(T),\iso)$, where we take $f:\omega\to X$ to $M_{\im(f)}$.  Certainly $M_{\im(f)}\iso M_{\im(g)}$ if and only if $\im(f)=\im(g)$, if and only if $fEg$.  So $\iso_2\borelleq \iso_T$, as desired.
	\end{proof}
	
	We summarize our findings in the following compilation theorem:
	
	\begin{theorem}\label{MainTheorem}
		Let $T$ be a CLO.  If $T$ is locally nonsimple, then 
		
		\begin{enumerate}
			\item $T$ is $\lambda$-Borel complete for all $\lambda$.
		\end{enumerate}
		
		Otherwise $T$ is locally simple and exactly one of the following happens:
		
		\begin{enumerate}
			\setcounter{enumi}{1}
			\item $T$ is $\aleph_0$-categorical.
			\item There is some $n$ with $3\leq n<\omega$ where $\iso_T\boreleq (n,=)$ and $I_{\infty\omega}(T)=n$.
			\item $\iso_T\boreleq \iso_1$ and $I_{\infty\omega}(T)=\beth_1$.
			\item $\iso_T\boreleq \iso_2$ and $I_{\infty\omega}(T)=\beth_2$.
		\end{enumerate}
		
		All five cases are possible, including every value of $n$ with $3\leq n<\omega$.
	\end{theorem}
	
	The possibility of each case is evidenced by the following examples.  We leave the verifications to the reader.  Note that here and subsequently, when we informally ``add a constant'' for an element $c$, we really mean we add a unary predicate $P_c$, which is true on $c$ and false elsewhere.
	
	\begin{enumerate}
		\item $(\b Z,<)$ is $\lambda$-Borel complete for all $\lambda$.
		\item $(\b Q,<)$ is $\aleph_0$-categorical.
		\item Let $3\leq n<\omega$, let $L=\{<\}\cup\{c_i:i\in\omega\}\cup\{P_i:0\leq i\leq n-3\}$.  Let $T_n$ state that $<$ is a dense linear without endpoints, that the $P_i$ are disjoint, dense, codense, and exhaustive, that $c_i<c_{i+1}$ for all $i$, and that $P_0(c_i)$ always holds.
		
		Then $T_n$ is complete and has exactly $n$ countable models.
		\item Let $S\subset\b Q$ be the set of all rationals of the form $k+\frac{1}{n}$ where $k\in\b Z$ and $2\leq n<\omega$. Then $(\b Q,<,c_s)_{s\in S}$ is Borel equivalent to $\iso_1$.
		\item $(\b Q,<,q)_{q\in \b Q}$ is Borel equivalent to $\iso_2$.
	\end{enumerate}
	
	We end with a nice corollary of our findings, a special case when the language is finite and the dichotomy is very sharp.  This result generalizes a result of Schirmann in \cite{SchirmannLinearOrders}, where a countable version of the same theorem was proven for linear orders without any colors.
	
	\begin{corollary}
		If $T$ is a CLO in a finite language $L$, then either $T$ is $\aleph_0$-categorical or is $\lambda$-Borel complete for all $\lambda$.
	\end{corollary}
	\begin{proof}
		If $T$ is locally nonsimple, or if $\IT(T)$ is finite, the corollary follows from Theorem~\ref{MainTheorem}.  So suppose, by way of contradiction, that there is a nonisolated $\Phi\in \IT(T)$. Let $(\Phi_n:n\in\omega)$ be a sequence from $\IT(T)$ limiting to $\Phi$.  Without loss of generality, we assume $\Phi_n<\Phi_{n+1}$ for all $n$.  Since $L$ is finite, every $\aleph_0$-categorical CLO in $L$ is finitely axiomatizable by Theorem~\ref{FiniteAxiomatizationTheorem}, and there are only finitely many such theories of any particular rank by Theorem~\ref{Aleph0CatCharacterizationThm}.  Thus, for every $n$, there is an $L$-formula $\sigma_n(x,y)$ stating ``$x<y$ and $[x,y]$ is not an $\aleph_0$-categorical CLO of rank at most $n$.''  For a moment, suppose the partial type $\Gamma(x,y)=\bigcup_n\sigma_n(x,y)\cup \Phi(x)\cup \Phi(y)$ is consistent.  Then any sufficiently saturated $\m S\models T$ realizes it at some pair $[a,b]$.  But then $[a,b]$ is not $\aleph_0$-categorical, despite being a convex subset of the $\aleph_0$-categorical structure $\Phi(\m S)$.  This will give us a contradiction to Proposition~\ref{Aleph0CatGeneralFactsProp}, assuming we can show $\Gamma$ is consistent.
		
		We show this by compactness.  So let $\Gamma_0\subset\Gamma$ be finite.  Then $\Gamma_0(a,b)$ says at most that $a<b$, that $[a,b]$ is not $\aleph_0$-categorical of rank at most $k$ for some $k$, and that there is a formula $\phi(x)$, contained in cofinitely many of the $\Phi_n$, such that both $a$ and $b$ satisfy $\phi$.  So pass to some sufficiently saturated $\m S\models T$, and let $b\in\m S$ realize $\Phi$.  Let $m$ be large enough that realizing $\Phi_m$ guarantees realizing $\phi$, and let $a\in\m S$ realize $a$.  For every $n<\omega$, there is a convex formula $\phi_n$ where $\Phi_i\proves\phi_n$ if and only if $i=n$.  By Lemma~\ref{RelativizationLemma}, there is a formula $\phi_n^{\#}(x)$ where for all $c\in [a,b]$, $[a,b]\models \phi_n^{\#}(c)$ if and only if $\m S\models\phi_n(c)$.  But if $m<n<n'<\omega$, then $\phi_n^{\#}$ and $\phi_{n'}^{\#}$ are disjoint definable subsets of $[a,b]$, meaning $[a,b]$ admits infinitely many inequivalent formulas, so is not $\aleph_0$-categorical.  Thus $(a,b)$ realize $\Gamma_0$, completing the proof.
	\end{proof}

	\bibliography{Citations}   % name your BibTeX data base
	
\end{document}